\documentclass{article}
\usepackage{amsmath,amssymb,amsbsy,amsthm}

\let\epsilon\varepsilon
\let\phi\varphi

\def\N{\mathbb N}

\def\S{\mathcal S}

\def\cl{\operatorname{cl}}

 \newtheorem{theorem}{Theorem}
 \newtheorem{lemma}{Lemma}
 
 \newtheorem{proposition}{Proposition}
 \newtheorem{corollary}{Corollary}

\begin{document}
%
%


\title{Uniform hypothesis testing for finite-valued stationary processes}
\author{Daniil Ryabko \thanks{INRIA Lille-Nord Europe, 40, avenue Halley,
59650 Villeneuve d'Ascq, France,\vspace{6pt} daniil@ryabko.net}}
\date{}
\maketitle
\begin{abstract}
Given a discrete-valued sample $X_1,\dots,X_n$ we wish to decide whether it was generated by a distribution 
belonging to a family $H_0$, or it was generated by a  distribution 
belonging to a family $H_1$. In this work we assume that all distributions are stationary ergodic, and do 
not make any further assumptions (e.g. no independence or mixing rate assumptions). 
We would like to have a test whose probability of error (both Type I and Type II) is uniformly bounded.
More precisely,  we require that for each $\epsilon$ there exist a sample size $n$ such that probability of error 
is upper-bounded by $\epsilon$ for samples longer than $n$.
We find some necessary and some sufficient conditions on $H_0$ and $H_1$  under which a consistent  test (with this notion of consistency)  exists. 
These conditions are topological, with respect to the topology of distributional distance.
\end{abstract}
\section{Introduction}
Given a sample $X_1,\dots,X_n$ (where $X_i$ are from a finite  alphabet $A$) which is known to be generated by a stationary ergodic process,
 we wish to decide whether it was generated by a distribution  belonging to a family $H_0$, versus it was generated by a  distribution 
belonging to a family $H_1$. 
Unlike most of the works on the subject, we do not 
assume that $X_i$ are i.i.d., but only make a much weaker assumption that
the distribution generating the sample is stationary ergodic. 

%
%

A test is a function that takes a sample and gives a  binary (possibly incorrect) answer:  either the sample
was generated by a distribution from $H_0$ or by a distribution from $H_1$. An answer $i\in\{0,1\}$ is correct if the sample is generated by a distribution that belongs to $H_i$.
Here we are concerned with characterizing those pairs of $H_0$ and $H_1$ for which consistent tests exist.

{\bf Consistency.} In this work we consider the following notion of consistency. 
 For two hypothesis $H_0$ and $H_1$, a test  is called {\em uniformly consistent}, 
if for any $\epsilon>0$ there is a sample size $n$ such that the {\em probability of error on a sample of size larger than $n$ is not greater
than $\epsilon$  if any distribution from $H_0\cup H_1$  is chosen to generate the sample}.
Thus, a uniformly consistent test provides performance guarantees for finite sample sizes.

\noindent  {\bf The results.}
Here we obtain some  topological conditions of the hypotheses for which consistent
tests exist, for the case of stationary ergodic distributions. 

A distributional distance between two process distributions \cite{Gray:88}  is defined as a
weighted sum of probabilities of all possible tuples $X\in A^*$, where $A$ is the alphabet and  the weights
are positive and have a finite sum. 

The test $\phi_{H_0,H_1}$ that we construct  is based on empirical estimates of distributional distance.
It outputs $0$ if the given sample is closer to 
the (closure of) $H_0$ than to the (closure of) $H_1$, and outputs 1 otherwise.
The main result is as follows.

\noindent{\bf Theorem.} Let $H_0, H_1\subset \mathcal E$, where $\mathcal E$ is the set of all stationary ergodic 
process distributions.   If, for each $i\in\{0,1\}$ the set $H_i$ has probability 1 with respect to ergodic decompositions of every element of $H_i$, 
then there is a  uniformly consistent test for $H_0$ against $H_1$. Conversely, if there is a uniformly consistent test for $H_0$ against $H_1$, 
then, for each $i\in\{0,1\}$, the set $H_{1-i}$ has probability 0 with respect  to ergodic decompositions of every element of $H_i$.

\noindent {\bf Prior work.} 
This work continuous our previous research \cite{Ryabko:101c,Ryabko:121c}, which provides similar necessary and sufficient conditions 
for the existence of a consistent test, for a weaker notion of {\em asymmetric} consistency: Type I error is uniformly bounded, while Type II error
is required to tend to 0 as the sample size grows. 

Besides that, there is of course a vast body of literature on hypothesis testing for i.i.d. (real- or discrete-valued) data (see e.g. \cite{Lehmann:86, Kendall:61}).
There is, however,  much less literature on hypothesis testing beyond i.i.d. or parametric models.
For a weaker notion of consistency, namely, requiring that the test should stabilize on the correct answer for a.e. realization of the process (under either $H_0$ or $H_1$), 
 \cite{Kieffer:93} constructs a consistent test for so-called
constrained finite-state model classes (including finite-state Markov and hidden Markov processes), against the general alternative of stationary ergodic processes. 
For the same  notion of consistency, 
 \cite{Nobel:06} gives sufficient conditions on two hypotheses $H_0$ and $H_1$ that consist of stationary ergodic real-valued processes, under which a consistent
test exists, extending the results of~\cite{Dembo:94} for i.i.d.\ data. The latter condition is that $H_0$ and $H_1$ are contained in disjoint $F_\sigma$ sets (countable unions of closed sets), with respect to the topology of weak convergence.
 Asymmetrically consistent tests for some specific hypotheses, but under the general alternative of stationary ergodic
processes, have been proposed in 
  \cite{BRyabko:06a,BRyabko:06b,Ryabko:08three, Ryabko:103s}, which address problems of testing identity, independence, estimating the order of a Markov process,
and also the change point problem. Noteworthy, a conceptually simple hypothesis of homogeneity (testing whether two sample are generated by the same
or by different processes) does not admit a consistent test even in the weakest asymptotic sense, as was shown in \cite{Ryabko:10discr}. Empirical estimates
of distributional distance have been also used to address the problem of clustering time series  \cite{Ryabko:10clust,Khaleghi:12}.

\section{Preliminaries}\label{s:prel}
Let $A$ be a finite alphabet, and denote  $A^*$ the set of words (or tuples) $\cup_{i=1}^\infty A^i$.
For a word  $B$ the symbol $|B|$ 
stands for the length of $B$.
Denote $B_i$ the $i$th
element of $A^*$, enumerated in such a way that the elements of
$A^{i}$ appear before the
elements of $A^{i+1}$, for all $i\in\N$.
{\em Distributions} or {\em  (stochastic) processes} are probability measures on the space $(A^\infty,\mathcal F_{A^\infty})$, where $\mathcal F_{A^\infty}$ is the 
Borel sigma-algebra of $A^\infty$.
%
Denote $\#(X,B)$ the number of occurrences of a word $B$ in a
word $X\in A^*$ and $\nu(X,B)$ its frequency:
$
\#(X,B)=\sum_{i=1}^{|X|-|B|+1} I_{\{(X_i,\dots,X_{i+|B|-1})=B\}},
$
and
\begin{equation}\label{eq:freq}
\nu(X,B)=\left\{\begin{array}{cc}{1\over|X|-|B|+1}\#(X,B) & \text{ if }|X|\ge|B|,\\ 0 &\text{
otherwise,}\end{array}\right.
\end{equation}
where $X=(X_1,\dots,X_{|X|})$.
For example, $\nu(0001,00)= 2/3.$

We use the abbreviation $X_{1..k}$ for $X_1,\dots,X_k$.
A process $\rho$ is {\em stationary}  if 
$$
\rho(X_{1..|B|}=B)=\rho(X_{t..t+|B|-1}=B)
$$
 for any $B\in A^*$ and $t\in\N$.
Denote  $\S$ the set of all stationary  processes on $A^\infty$.
A stationary process $\rho$ is called {\em (stationary) ergodic} if
   the frequency of occurrence of each word
$B$ in a sequence $X_1,X_2,\dots$ generated by $\rho$ tends to its
a priori (or limiting) probability a.s.: 
$\rho(\lim_{n\rightarrow\infty}\nu(X_{1..n},B)= \rho(X_{1..|B|}=B))=1$. 
 Denote $\mathcal E$ the set  of all stationary ergodic processes.

A {\bf distributional distance}  is defined for a pair of processes
$\rho_1,\rho_2$ as follows~\cite{Gray:88}:
$$
d(\rho_1,\rho_2)=\sum_{i=1}^\infty w_i |\rho_1(X_{1..|B_i|}=B_i)-\rho_2(X_{1..|B_i|}=B_i)|,
$$
where $w_i$ are summable positive real weights (e.g. $w_k=2^{-k}$: we fix this choice for the sake of concreteness).
It is easy to see that $d$ is a metric.
Equipped with this metric, the space of all stochastic processes is a compact, and the set of stationary 
processes  $\S$ is its convex closed subset. 
(The set $\mathcal E$ is not closed.)
When talking about closed and open subsets of $\S$ we assume the topology of~$d$.
Compactness of the set $\S$ is one of the main ingredients in the proofs of the main results.
Another is that the distance $d$ can be consistently estimated, as the following lemma shows (because of its importance for
further development, we give it with a proof). 
\begin{lemma}[$\hat d$ is consistent  \cite{Ryabko:08three,Ryabko:103s}\ ]\label{th:dd} Let $\rho,\xi\in\mathcal E$ and let a sample $X_{1..k}$
be generated by $\rho$.  Then
$$
\lim_{k\rightarrow\infty}\hat d(X_{1..k},\xi)=d(\rho,\xi)\ \rho\text{-a.s.}
$$
\end{lemma}

\begin{proof}
For any $\epsilon>0$ find such an index $J$ that
$\sum_{i=J}^\infty w_i<\epsilon/2$.
For each $j$ we have $\lim_{k\rightarrow\infty}\nu(X_{1..k},B_j)= \rho(B_j)$
a.s., so that
$|\nu(X_{1..k},B_j) - \rho(B_j)|<\epsilon/(2Jw_j)$ from some $k$ on; denote  $K_j$ this $k$.
 Let
$K=\max_{j<J}K_j$ ($K$ depends
on the realization $X_1,X_2,\dots$).
 Thus, for $k>K$ we have
\begin{multline*}
   |\hat d(X_{1..k},\xi) - d(\rho,\xi)| 
= \left|\sum_{i=1}^\infty w_i\big(|\nu(X_{1..k},B_i)-\xi(B_i)| - |\rho(B_i)-\xi(B_i)| \big) \right|
  \\
   \le \sum_{i=1}^\infty w_i|\nu(X_{1..k},B_i)-\rho(B_i)|  
    \le \sum_{i=1}^J w_i|\nu(X_{1..k},B_i)-\rho_X(B_i)|  +\epsilon/2 
   \\
   \le \sum_{i=1}^Jw_i \epsilon/(2Jw_i) +\epsilon/2 = \epsilon,
\end{multline*}
which proves the statement.
\end{proof}

Considering 
the Borel (with respect to the metric $d$) sigma-algebra $\mathcal F_\S$ on the set $\S$, we obtain
a standard probability space $(\S,\mathcal F_\S)$.
An important tool that will be used in the analysis is  {\bf ergodic decomposition} of stationary processes  (see e.g. \cite{Gray:88, Billingsley:65}):
 any stationary process can be expressed as a mixture of stationary ergodic
processes. 
More formally, for any $\rho\in\S$ there is a measure $W_\rho$ on $(\S,\mathcal F_\S)$, such 
that $W_\rho(\mathcal E)=1$, and $\rho(B)=\int d W_\rho(\mu)\mu(B)$, for any $B\in\mathcal F_{A^\infty}$.
The {\em support} of a stationary distribution $\rho$ is the minimal closed set $U\subset\S$ such that $W_\rho(U)=1$.

A {\bf test} is a function $\phi: A^*\rightarrow\{0,1\}$ that takes a sample and outputs a binary answer,
where the answer $i$ is interpreted as ``the sample was generated by a distribution that belongs to $H_i$''.
The answer $i$ is correct if the sample was indeed generated by a distribution from $H_i$, otherwise
we say that the test made an {\bf error}. 

A test $\phi$ is called {\bf uniformly consistent} if for every $\alpha$ there is an $n_\alpha\in\N$ such that
for every $n\ge n_\alpha$ the probability of error on a sample of size $n$ is less than $\alpha$: $\rho(X\in A^n: \phi(X)=i) <\alpha$ for every $\rho\in H_{1-i}$ and every $i\in\{0,1\}$.


\section{Main results}\label{s:main}
The tests presented below are based on {\em empirical estimates of  the distributional distance $d$}:
$$
\hat d(X_{1..n},\rho)=\sum_{i=1}^\infty w_i |\nu(X_{1..n},B_i)-\rho(B_i)|,
$$
where $n\in\N$, $\rho\in\S$, $X_{1..n}\in A^n$. That is, $\hat d(X_{1..n},\rho)$ measures
the discrepancy between empirically estimated and theoretical probabilities.
 For a sample $X_{1..n}\in A^n$ and a hypothesis $H\subset\mathcal E$ define 
$$
\hat d(X_{1..n},H)=\inf_{\rho\in H} \hat d(X_{1..n},\rho).
$$

For $H\subset\S$, denote $\cl H$ the closure of $H$ (with respect to the topology of $d$).

For $H_0,H_1\subset\S$, the {\bf uniform test} $\phi_{H_0,H_1}$ is constructed   as follows. For each $n\in\N$ let
\begin{equation}
 \phi_{H_0,H_1}(X_{1..n}):=\left\{\begin{array}{ll} 0& \text{ if }  \hat d(X_{1..n},\cl H_0\cap\mathcal E)<\hat d(X_{1..n},\cl H_1\cap\mathcal E), \\ 1 & \text{ otherwise.} \end{array}\right.
\end{equation}


\begin{theorem}[uniform testing]\label{th:uni}
 Let $H_0\subset\S$ and $H_1\subset\S$.  If 
$W_\rho(H_i)=1$ for every $\rho\in \cl H_i$
then 
the test $\phi_{H_0,H_1}$ is uniformly consistent.
Conversely, if there exists a uniformly consistent test for $H_0$ against $H_1$ then 
$W_\rho(H_{1-i})=0$ for any $\rho\in cl H_i$.
\end{theorem}
The proof is deferred to section~\ref{s:proofs}.

\section{Examples}\label{s:ex}
 First of all, it is obvious that sets that consist of just
one or finitely many stationary ergodic processes are closed and closed under ergodic decompositions; therefore, 
for any pair of disjoint sets of this type, there exists a uniformly consistent test. (In particular, there 
is a uniformly consistent test for $H_0=\{\rho_0\}$ against $H_1=\{\rho_1\}$, where $\rho_0,\rho_1\in\mathcal E$.)
 
It is clear that  for any $\rho_0$ 
 there is no uniformly consistent test for $\{\rho_0\}$ against $\mathcal E\backslash\{\rho_0\}$.    
 More generally, for any non-empty $H_0$ there is no uniformly 
consistent test for $H_0$ against $\mathcal E\backslash H_0$ provided the latter complement is also non-empty.
Indeed, this follows from Theorem~\ref{th:uni}  since in these cases the closures of $H_0$ and $H_1$ are not disjoint.
One might suggest at this point that a uniformly consistent test exists if we restrict $H_1$ to those processes
that are sufficiently far from $\rho_0$. However, this is not true. We can prove an even stronger negative result.
\begin{proposition}\label{th:no}
Let  $\rho,\nu\in\mathcal E$, $\rho\ne\nu$ and let  $\epsilon>0$.
There is no uniformly consistent test for $H_0=\{\rho\}$ against $H_1=\{\nu'\in\mathcal E: d(\nu',\nu)\le\epsilon\}$. 
\end{proposition}
The proof of the proposition is deferred to the next section. What the proposition means is that, while distributional distance
is well suited for characterizing those hypotheses for which consistent test exist, it is not 
suited for {\em formulating the actual hypotheses}. Apparently a stronger distance is needed for the latter.

The following statement is easy to demonstrate from Theorem~\ref{th:uni}.
\begin{corollary}
  Given two disjoint sets  $H_0$ and $H_1$ each of which 
is continuously parametrized by a compact set of parameters and is closed under taking ergodic decompositions, 
there exists a uniformly consistent test of $H_0$ against $H_1$.
\end{corollary}
Examples of parametrisations mentioned in the Corollary are the sets of $k$-order Markov sources, parametrised by 
transition probabilities. Thus, any two disjoint closed subsets of these sets satisfy the assumption of the Corollary.

\section{Proofs}\label{s:proofs}
The proof of Theorem~\ref{th:uni} will use the following lemmas, whose proofs can be found in \cite{Ryabko:121c}.
\begin{lemma}[smooth probabilities of deviation]\label{th:count}
 Let $m>2k>1$, $\rho\in\S$, $H\subset\S$, and $\epsilon>0$. Then
\begin{equation}\label{eq:count1} 
\rho(\hat d(X_{1..m},H)\ge\epsilon)  
  \le 2\epsilon'^{-1} \rho(\hat d(X_{1..k},H)\ge \epsilon'),
\end{equation}
where $\epsilon':=\epsilon - \frac{2k}{m-k+1} - t_k$ with $t_k$ being the sum of  all the weights of tuples longer than $k$ in the definition of $d$: $t_k:=\sum_{i:|B_i|>k}w_i$. Further,
\begin{equation}\label{eq:count2} 
\rho(\hat d(X_{1..m}, H)\le\epsilon)
  \le 2\rho\left(\hat d(X_{1..k},H)\le \frac{m}{m-k+1}2\epsilon + \frac{4k}{m-k+1}\right).
\end{equation}
\end{lemma}
The meaning of this lemma is as follows.  For any word $X_{1..m}$, if it is far away from (or close to)
a given distribution $\mu$ (in the empirical distributional distance), then some of its shorter subwords $X_{i..i+k}$ are far from (close to) $\mu$ too.
 In other words, for a stationary distribution $\mu$,  it cannot happen that  a small 
 sample is likely to be close to $\mu$, but a larger sample is likely to be far.

\begin{lemma}\label{th:cont} Let $\rho_k\in\S$, $k\in\N$ be a sequence of processes that converges to a process $\rho_*$.
Then, for any $T\in A^*$ and $\epsilon>0$ if $\rho_k(T)>\epsilon$ for infinitely many indices $k$, then
$\rho_*(T)\ge\epsilon$.
\end{lemma}
 This statement follows from the fact that  $\rho(T)$ is continuous as a function of $\rho$. 

\smallskip
{\em Proof of Theorem~\ref{th:uni}.}
To prove the first statement of the theorem,
we will show that the test $\phi_{H_0,H_1}$ is a uniformly consistent test for $\cl H_0\cap\mathcal E$ against $\cl H_1\cap\mathcal E$ (and hence for $H_0$ against $H_1$),
under  the conditions of the theorem. 
Suppose that, on the contrary,  for some $\alpha>0$ for every $n'\in\N$ there is a process $\rho\in \cl H_0$ such that 
$\rho(\phi(X_{1..n})=1)>\alpha$ for some $n>n'$. 
Define 
$$
\Delta:=d(\cl H_0,\cl H_1):=\inf_{\rho_0\in \cl H_0\cap\mathcal E, \rho_1\in \cl H_1\cap\mathcal E} d(\rho_0,\rho_1),
$$ which is positive since $\cl H_0$ and $\cl H_1$ are closed and disjoint.
We have
\begin{multline}\label{eq:union}
\alpha<\rho(\phi(X_{1..n})=1)\\ \le   \rho(\hat d(X_{1..n},H_0)\ge\Delta/2\ or \ \hat d(X_{1..n},H_1)<\Delta/2)\\ \le \rho(\hat d(X_{1..n},H_0)\ge\Delta/2) + \rho(\hat d(X_{1..n},H_1)<\Delta/2).
\end{multline} 
This implies that either $\rho(\hat d(X_{1..n},\cl H_0)\ge\Delta/2)>\alpha/2$ or $\rho(\hat d(X_{1..n},\cl H_1)<\Delta/2)>\alpha/2$,
so that, by assumption, at least one of these inequalities holds for infinitely many $n\in\N$ for some sequence  $\rho_n\in H_0$.
Suppose that it is the first one, that is, there is an increasing sequence $n_i$, $i\in\N$ and a sequence $\rho_i\in\cl H_0$, $i\in\N$ 
such that 
\begin{equation}\label{eq:da}
\rho_i(\hat d(X_{1..n_i},\cl H_0)\ge\Delta/2)>\alpha/2 \text{ for all }i\in\N. 
\end{equation}
The set $\S$ is compact, hence so is its closed subset $\cl H_0$. Therefore, the sequence $\rho_i$, $i\in\N$ must
contain a subsequence that converges to a certain process $\rho_*\in\cl H_0$. Passing to a subsequence if necessary, we may assume
that this convergent subsequence is the sequence $\rho_i$, $i\in\N$ itself.

Using Lemma~\ref{th:count}, (\ref{eq:count1}) (with $\rho=\rho_{n_m}$, $m=n_m$, $k=n_k$, and $H=\cl H_0$),  and taking  $k$  large enough to have $t_{n_k}<\Delta/4$,
 for every $m$  large enough to have $\frac{2n_k}{n_m-n_k+1}<\Delta/4$, we obtain 
\begin{equation}\label{eq:down}
8\Delta^{-1}\rho_{n_m}\left(\hat d(X_{1..n_k},\cl H_0)  \ge\Delta/4\right)  \ge  \rho_{n_m}\left(\hat d(X_{1..n_m},\cl H_0)\ge\Delta/2\right) > \alpha/2. 
\end{equation}
That is, we have shown that for any large enough index $n_k$ the inequality 
$\rho_{n_m}(\hat d(X_{1..n_k},\cl H_0)\ge\Delta/4)> \Delta\alpha/16$ holds for 
infinitely many indices $n_m$. 
From this and Lemma~\ref{th:cont} with  $T=T_k:=\{X:\hat d(X_{1..n_k},\cl H_0)\ge\Delta/4\}$  we conclude 
that $\rho_*(T_k)>\Delta\alpha/16$.
The latter holds for infinitely many $k$; that is,  $\rho_*(\hat d(X_{1..n_k},\cl H_0)\ge\Delta/4)>\Delta\alpha/16$ infinitely often.
Therefore, 
$$
\rho_*(\limsup_{n\rightarrow\infty} d(X_{1..n},\cl H_0)\ge\Delta/4)>0.
$$ However, we must have 
$$
\rho_*(\lim_{n\rightarrow\infty} d(X_{1..n},\cl H_0)=0)=1
$$
for every $\rho_*\in\cl H_0$: indeed, for $\rho_*\in\cl H_0\cap\mathcal E$ it follows from Lemma~\ref{th:dd}, and for $\rho_*\in\cl H_0\backslash\mathcal E$ 
from  Lemma~\ref{th:dd}, ergodic decomposition and the conditions of the theorem. 

Thus, we have arrived at a contradiction that shows that $\rho_n(\hat d(X_{1..n},\cl H_0)>\Delta/2)>\alpha/2$ cannot hold for 
infinitely many $n\in\N$ for any sequence of  $\rho_n\in\cl H_0$. Analogously, we can show that $\rho_n(\hat d(X_{1..n},\cl H_1)<\Delta/2)>\alpha/2$
cannot hold for infinitely many $n\in\N$ for any sequence of $\rho_n\in\cl H_0$. Indeed,  using Lemma~\ref{th:count}, equation~(\ref{eq:count2}), we can show that
  $\rho_{n_m}(\hat d(X_{1..n_m},\cl H_1)\le\Delta/2) > \alpha/2$ for a large enough $n_m$
implies $\rho_{n_m}(\hat d(X_{1..n_k},\cl H_1)\le 3\Delta/4)> \alpha/4$ for a smaller $n_k$.
Therefore, if we assume that $\rho_n(\hat d(X_{1..n},\cl H_1)<\Delta/2)>\alpha/4$ for infinitely many $n\in\N$ for some sequence of $\rho_n\in\cl H_0$, then 
we will also find a $\rho_*$ for which $\rho_*(\hat d(X_{1..n},\cl H_1)\le 3\Delta/4)> \alpha/4$ for infinitely
many $n$, which, using Lemma~\ref{th:dd} and ergodic decomposition, can be shown to contradict the fact that $\rho_*(\lim_{n\rightarrow\infty}  d(X_{1..n},\cl H_1)\ge\Delta)=1$.
 
Thus, returning to~(\ref{eq:union}), we have shown that from some $n$ on there is no $\rho\in \cl H_0$ for which $\rho(\phi=1)>\alpha$
holds true. The statement for $\rho\in\cl H_1$ can be proven analogously, thereby finishing the proof of the first statement.

To prove the second statement of the theorem, 
we assume that there exists a uniformly consistent test $\phi$ for $H_0$ against $H_1$, and we will show that
$W_\rho(H_{1-i})=0$ for every $\rho\in \cl H_i$.
Indeed, let $\rho\in \cl H_0$, that is, suppose that there is a sequence $\xi_i\in H_0, i\in\N$ such that $\xi_i\to\rho$. 
Assume  $W_\rho(H_1)=\delta>0$ and take $\alpha:=\delta/2$.
Since the test $\phi$ is uniformly consistent, there is an $N\in\N$ such  that  for every $n>N$ we have  
\begin{multline*}
 \rho(\phi(X_{1..n}=0))\le \int_{H_1} \phi(X_{1..n}=0)dW_\rho + \int_{\mathcal E\backslash H_1} \phi(X_{1..n}=0) dW_\rho \\ \le \delta\alpha + 1-\delta \le 1-\delta/2.
\end{multline*}
Recall that, for $T\in A^*$,  $\mu(T)$ is a continuous function in $\mu$. In particular, this holds for the set $T=\{X\in A^n: \phi(X)=0\}$, 
for any given $n\in\N$. Therefore, for every  $n>N$ and for every  $i$ large enough,  $\rho_i(\phi(X_{1..n})=0)<1-\delta/2$ implies also $\xi_i(\phi(X_{1..n})=0)<1-\delta/2$
which contradicts $\xi_i\in H_0$.
This contradiction shows $W_\rho(H_1)=0$ for every $\rho\in \cl H_0$. The case $\rho\in \cl H_1$ is analogous.
The theorem is proven.

\smallskip
{\em Proof of Proposition~\ref{th:no}.}\ 
Assume $d(\rho,\nu)>\epsilon$ (the other case is obvious).
Consider the process $(x_1,y_1),(x_2,y_2),\dots$ on pairs $(x_i,y_i)\in A^2$, such that
the distribution of $x_1,x_2,\dots$ is $\nu$, the distribution of $y_1,y_2,\dots$ is $\rho$ and 
the two components $x_i$ and $y_i$ are independent; in other words, the distribution of $(x_i,y_i)$ is $\nu\times\rho$. 
Consider also a two-state stationary ergodic  Markov chain $\mu$, with two states $1$ and $2$, whose transition
probabilities are $\left(\begin{array}{cc}
                          1-p & p \\ q & 1-q
                         \end{array} \right)$, where $0<p<q<1$. 
The limiting (and initial) probability of the state $1$ is $p/(p+q)$ and that of the state $2$ is $q/(p+q)$. 
Finally, the process $z_1,z_2,\dots$ is constructed as follows:
$z_i=x_i$ if $\mu$ is in the state $a$ and $z_i=y_i$ otherwise (here it is assumed that the chain $\mu$ generates
a sequence of outcomes independently of $(x_i,y_i)$. Clearly, for every $p,q$ satisfying $0<p<q<1$ the process $z_1,z_2,\dots$ is stationary ergodic; denote $\zeta$ its distribution.
Let $p_n:=1/(n+1)$, $n\in\N$.  Since $d(\rho,\nu)>\epsilon$, we can find a $\delta>0$  such 
that   $d(\rho,\zeta_n)>\epsilon$   
where $\zeta_n$ is the distribution $\zeta$ with parameters $p_n$ and $q_n$, where $q_n$ satisfies $q_n/(p_n+q_n)=\delta$. 
Thus, $\zeta_n\in H_1$ for all $n\in\N$. However, $\lim_{n\to\infty}\zeta_n=\zeta_\infty$
where $\zeta_\infty$ is the stationary distribution with $W_{\zeta_\infty}(\rho)=\delta$ and $W_{\zeta_\infty}(\nu)=1-\delta$.
Therefore, $\zeta_\infty\in\cl H_1$ and $W_{\zeta_\infty}(H_0)>0$, so that by Theorem~\ref{th:uni} there is no uniformly consistent
test for $H_0$ against $H_1$, which concludes the proof.

\subsection*{Acknowledgements}  This work has been partially supported by
    the French Ministry of Higher Education and Research, Nord-Pas de Calais Regional Council and FEDER through the 
``Contrat de Projets Etat Region (CPER) 2007-2013.'' 


\begin{thebibliography}{10}

\bibitem{Billingsley:65}
P.~Billingsley.
\newblock {\em Ergodic theory and information}.
\newblock Wiley, New York, 1965.

\bibitem{Dembo:94}
A.~Dembo and Y.~Peres.
\newblock A topological criterion for hypothesis testing.
\newblock {\em Ann. Math. Stat.}, 22:106--117, 1994.

\bibitem{Gray:88}
R.~Gray.
\newblock {\em Probability, Random Processes, and Ergodic Properties}.
\newblock Springer Verlag, 1988.

\bibitem{Kendall:61}
M.G. Kendall and A.~Stuart.
\newblock {\em The advanced theory of statistics; Vol.2: Inference and
  relationship}.
\newblock London, 1961.

\bibitem{Khaleghi:12}
A.~Khaleghi, D.~Ryabko, J.~Mary, and P.~Preux.
\newblock Online clustering of processes.
\newblock In {\em AISTATS}, JMLR W\&CP 22, pages 601--609, 2012.

\bibitem{Kieffer:93}
J.C. Kieffer.
\newblock Strongly consistent code-based identification and order estimation
  for constrained finite-state model classes.
\newblock {\em IEEE Transactions on Information Theory}, 39(3):893--902, 1993.

\bibitem{Lehmann:86}
E.~Lehmann.
\newblock {\em Testing Statistical Hypotheses, 2nd edition}.
\newblock Wiley, New York, 1986.

\bibitem{Nobel:06}
A.~Nobel.
\newblock Hypothesis testing for families of ergodic processes.
\newblock {\em Bernoulli}, 12(2):251--269, 2006.

\bibitem{BRyabko:06a}
B.~Ryabko and J.~Astola.
\newblock Universal codes as a basis for time series testing.
\newblock {\em Statistical Methodology}, 3:375--397, 2006.

\bibitem{BRyabko:06b}
B.~Ryabko, J.~Astola, and A.~Gammerman.
\newblock Application of {K}olmogorov complexity and universal codes to
  identity testing and nonparametric testing of serial independence for time
  series.
\newblock {\em Theoretical Computer Science}, 359:440--448, 2006.

\bibitem{Ryabko:10clust}
D.~Ryabko.
\newblock Clustering processes.
\newblock In {\em Proc. the 27th International Conference on Machine Learning
  (ICML 2010)}, pages 919--926, Haifa, Israel, 2010.

\bibitem{Ryabko:10discr}
D.~Ryabko.
\newblock Discrimination between {B}-processes is impossible.
\newblock {\em Journal of Theoretical Probability}, 23(2):565--575, 2010.

\bibitem{Ryabko:101c}
D.~Ryabko.
\newblock Testing composite hypotheses about discrete-valued stationary
  processes.
\newblock In {\em Proc. IEEE Information Theory Workshop (ITW'10)}, pages
  291--295, Cairo, Egypt, 2010. IEEE.

\bibitem{Ryabko:121c}
D.~Ryabko.
\newblock Testing composite hypotheses about discrete ergodic processes.
\newblock {\em Test}, 21(2):317--329, 2012.

\bibitem{Ryabko:08three}
D.~Ryabko and B.~Ryabko.
\newblock On hypotheses testing for ergodic processes.
\newblock In {\em Proc. 2008 IEEE Information Theory Workshop}, pages 281--283,
  Porto, Portugal, 2008. IEEE.

\bibitem{Ryabko:103s}
D.~Ryabko and B.~Ryabko.
\newblock Nonparametric statistical inference for ergodic processes.
\newblock {\em IEEE Transactions on Information Theory}, 56(3):1430--1435,
  2010.

\end{thebibliography}

\end{document}